\documentclass{article}
\usepackage{amssymb}
\usepackage{mathrsfs}
\usepackage[utf8]{inputenc}
\usepackage{tikz-cd}
\usetikzlibrary{cd}
\usepackage{amsmath,amsfonts}

\newcommand{\im}{\operatorname{im}}

\newcommand{\zz}{\mathbb{Z}}

\newcommand{\ff}{\mathbb{F}}

\newtheorem{lemma}{Lemma}[section]
\newtheorem{proposition}{Proposition}[section]
\newtheorem{theorem}{Theorem}
\newtheorem{lettertheorem}{Theorem}

\newtheorem{definition}{Definition}

\newtheorem{example}{Example}[section]
\newtheorem{remark}{Remark}[section]
\newtheorem{corollary}{Corollary}[section]
\newenvironment{proof}{\vspace*{2ex}\noindent {\em Proof:}
}{\hfill $\diamond$ \\[2ex]}

\usepackage[colorinlistoftodos,prependcaption,textsize=scriptsize]{todonotes}

\makeatletter
\providecommand\@dotsep{5}
\renewcommand{\listoftodos}[1][\@todonotes@todolistname]{%
  \@starttoc{tdo}{#1}}
\makeatother

\title{Gabriel's Theorem for Locally Finite-Dimensional Representations of Infinite Quivers}
\author{Nathaniel Gallup and Stephen Sawin}
\date{\today}
\begin{document}

\maketitle

\abstract{We prove a version of Gabriel's theorem for locally finite-dimensional representations of infinite quivers. Specifically, we show that if $\Omega$ is any connected quiver, the category of locally finite-dimensional representations of $\Omega$ has unique representation type (meaning no two indecomposable representations have the same dimension vector) if and only if the underlying graph of $\Omega$ is a generalized ADE Dynkin diagram (i.e. one of $A_n, D_n, E_6, E_7, E_8, A_{\infty}, A_{\infty , \infty}$ or $D_\infty$). This result is companion to earlier work of the authors generalizing Gabriel's theorem to infinite quivers with different conditions.}

\section{Introduction}

In \cite{gabriel1972}, Gabriel showed that a quiver $\Omega$ has finite representation type (meaning the abelian category $\text{rep}(\Omega)$ of finite-dimensional representations of $\Omega$ has only finitely many indecomposable objects) if and only if the underlying graph of $\Omega$ is an ADE Dynkin diagram. He showed moreover that dimension vectors give a bijection between the set of indecomposable representations of $\Omega$ and the positive roots of $\Omega$, (i.e. the integer-valued functions on the vertex set of $\Omega$ which have Euler-Tits form equal to $1$). One consequence of this is that $\Omega$ has \emph{unique representation type}, meaning no two indecomposable representations can have the same dimension vector. In fact it is a consequence of Gabriel's theorem that for finite quivers $\Omega$, this condition is equivalent to $\Omega$ having finite representation type. Note also that for any quiver $\Omega$, $\text{rep}(\Omega)$ is always infinite Krull-Schmidt (meaning every representation is a direct sum of indecomposable representations). This was proved for several classes of quivers, including $A_{\infty , \infty}$ quivers, in \cite{bautista-liu-paquette2011} and for any quiver in \cite{botnam2017zigzag} and \cite{botnam2020persistence}. 

In \cite{gallup2023gabriel} we gave a version of Gabriel's theorem for infinite quivers. Precisely, we showed that the category $\text{Rep}(\Omega)$ of \emph{all} (possibly infinite-dimensional) representations of $\Omega$ has unique representation type and is infinite Krull-Schmidt (meaning every representation is a direct sum of possibly infinitely many indecomposable representations) if and only if $\Omega$ is an eventually outward generalized ADE Dynkin quiver (see Figure \ref{fig: generalized ade quivers}) and moreover in this case dimension vectors give a bijection between the set of indecomposable representations and the positive roots of $\Omega$, defined relative to an infinite Euler-Tits form. 

\begin{figure}\label{fig: generalized ade quivers} 
\begin{tikzcd}[arrows=-]
   A_n \quad = \qquad    \bullet \arrow[r] &   \bullet \arrow[r]  & \quad \cdots\quad \arrow[r]& \bullet \arrow[r]&  \bullet
 \end{tikzcd}
 
 \begin{tikzcd}[arrows=-]
 & & & & & \bullet \arrow[ld] \\
 D_n \quad = \qquad \bullet \arrow[r] & \bullet \arrow[r] &\quad \cdots \arrow[r] &\bullet \arrow[r] &
  \bullet  & \\
 & & & & &\bullet \arrow[lu]
\end{tikzcd}

 \begin{tikzcd}[arrows=-]
 & & \bullet \arrow[d] & & & \\
 E_m \quad = \qquad  \bullet \arrow[r] & \bullet \arrow[r] & \bullet \arrow[r] &  \bullet \arrow[r] & \quad \cdots \quad \arrow[r] &
  \bullet  
\end{tikzcd}

 \begin{tikzcd}[arrows=-]
   A_\infty\quad  = \qquad  \bullet \arrow[r] &   \bullet \arrow[r]  &  \bullet \arrow[r]  &\quad \cdots 
 \end{tikzcd}
 
 \begin{tikzcd}[arrows=-]
   A_{\infty,\infty} \quad  = \qquad  \cdots \quad \arrow[r] & \bullet \arrow[r]  &   \bullet \arrow[r] & \bullet \arrow[r] & \quad \cdots 
 \end{tikzcd}

 \begin{tikzcd}[arrows=-]
 & & & & \bullet \arrow[ld] \\
 D_\infty \quad = \qquad \cdots \quad \arrow[r] & \bullet \arrow[r] & \bullet \arrow[r] &
  \bullet  & \\
 & & & & \bullet \arrow[lu]
\end{tikzcd}
\caption{The Generalized ADE Dynkin Quivers, where $m = 6, 7 , 8$}
\end{figure}
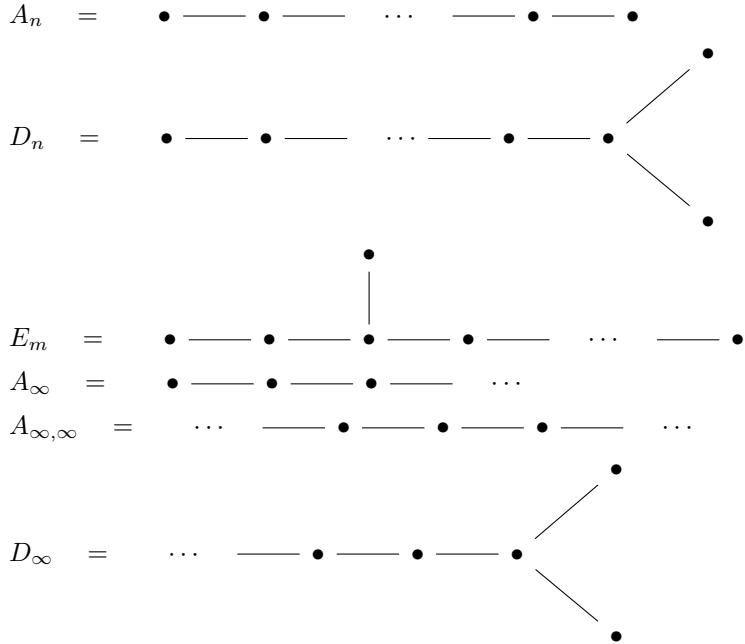

In fact we show that any quiver which is not eventually outward has a (necessarily infinite-dimensional) representation which is not the direct sum of indecomposable representations. The question remains, what happens in the category $\text{rep}(\Omega)$ of locally finite-dimensional representations of an infinite quiver? 

Indeed in \cite[Section 9]{gallup2023gabriel}, we conjectured that for \emph{any} (not necessarily eventually outward) quiver $\Omega$, $\text{rep}(\Omega)$ has unique representation type if and only if $\Omega$ is a generalized ADE Dynkin quiver. In this paper we prove our conjecture, stated formally below. 

\begin{lettertheorem}[Locally Finite-Dimensional Infinite Gabriel's Thm.]
Let $\Omega$ be a connected quiver. The category $\operatorname{rep}(\Omega)$ has unique representation type if and only if $\Omega$ is a generalized ADE Dynkin quiver (see Figure \ref{fig: generalized ade quivers}) and in this case, taking dimension vectors gives a bijection between the set of isomorphism classes of indecomposable representations and the positive roots of $\Omega$ (as defined in \cite{gallup2023gabriel}). 
\end{lettertheorem}

In Section \ref{sec: background} we provide the background on quivers and their representations that will be used in this paper. Then in Section \ref{sec: decomp of a infinity}, we show that any locally finite-dimensional representation of any $A_{\infty, \infty}$ quiver (regardless of the orientation of the arrows) is a direct sum of indecomposable representations. 

As a consequence of this, in Section \ref{sec: tails} we show that if $V$ is a locally finite-dimensional indecomposable representation of a quiver, then the dimension of $V$ cannot increase along a tail (a subquiver isomorphic to $A_{\infty}$) and as a consequence, such a representation must be FLEI (meaning all but finitely many of the arrows map to isomorphisms) if $\Omega$ is a finitely branching tree quiver. Finally in Section \ref{sec: proof} we prove the conjecture. 


\section{Background}\label{sec: background}

In this section we give the necessary background on quivers and their representations. 

\subsection{Quivers}

A quiver $\Omega=(\Omega_0, \Omega_1)$ is a set $\Omega_0$ of \emph{vertices} and a set $\Omega_1$ of ordered pairs of vertices called \emph{arrows}. If $a=(x,y) \in \Omega_1$ we say $s(a)=x$ is the source of $a$ and $t(a)=y$ is the target of $a$. In this case we sometimes say that $a$ points $x \to y$ or write $a: x \to y$. If either $a = (x , y)$ or $a = (y , x)$ then we say that $a$ is an arrow \emph{between $x$ and $y$}.

A \emph{journey} $p$ is a quiver morphism from $A_n$ or $A_\infty$ to $\Omega$ which is injective on vertices. It is called \emph{finite} if its domain is $A_n$ and \emph{infinite} if its domain is $A_\infty$. We call the image of the left-most vertex the \emph{source} of the journey, denoted by $s(p)$, and if the journey is finite, we call the image of the right-most vertex the \emph{target}, denoted by $t(p)$. A journey is called a \emph{cycle} if its domain is $A_n$ for $n \geq 2$ and $s(p) = t(p)$. We say that $\Omega$ is a \emph{tree} if it contains no cycles. We say that $\Omega$ is \emph{finitely-branching} if it is union of finitely many journeys. It is called \emph{eventually outward} if every journey contains at most finitely-many arrows that point towards the source of those journeys.

A \emph{tail} of a quiver is an infinite journey with vertices $x_0, x_1, \ldots$ such that in $\Omega$ there is exactly one arrow between $x_i$ and $x_{i + 1}$ and no arrows between $x_i$ and $x_j$ for $j \neq i + 1, i - 1$. An arrow on a tail of a quiver, say between $x_i$ and $x_{i + 1}$ \emph{points in the direction of the tail} if it points $x_i \to x_{i + 1}$ and it \emph{points in the direction opposite the tail} if it points $x_{i + 1} \to x_i$. 

\subsection{Representations of Quivers}

We fix an arbitrary field $\ff$ throughout the paper. A \emph{representation} $V$ of a quiver $\Omega$ is an assignment to every $x \in \Omega_0$ an $\ff$ vector space $V(x)$ and to every arrow $a : x \to y$ a linear transformation $V(a) : V(x) \to V(y)$. A \emph{subrepresentation} of $V$ is an assignment of a subspace $W(x)$ to every $x \in \Omega_0$ such that for all arrows $a: x \to y$, $V(a)(W(x)) \subseteq W(y)$, which implies $W$ is itself a representation of $\Omega$ with maps obtained by restricting those of $V$. If $\{ W_\alpha \mid \alpha \in A \}$ is a family of subrepresentations of $V$ we say that $V = \bigoplus_{\alpha \in A} W_\alpha$ if for all $x \in \Omega_0$ we have $V(x) = \bigoplus_{\alpha \in A} W_\alpha(x)$.

We say that $V$ is \emph{locally finite-dimensional} if $V(x)$ is a finite-dimensional vector space for all $x \in \Omega_0$.

\subsection{The Transport of a Subspace}

If we have vector spaces $V$ and $W$ and a linear map either $T: V \to W$ or $T: W \to V$ (in which case we say $T$ is a map \emph{between $V$ and $W$}) and $A \subseteq V$, the \emph{transport $A_T$} of $A$ to $W$ along $T$ is $T(A)$ in the first case and $T^{-1}(A)$ in the second. If $p$ is a journey in a quiver $\Omega$ with vertices $x_0, x_1, \ldots$, $V$ is a representation of $\Omega$, and $U \subseteq V(x_0)$ is a subspace, then the \emph{transport of $U$ along $p$ to $x_n$} is obtained by transporting $U$ along each of the maps of the journey to $x_n$. 

If $U$ and $W$ are vector spaces, $A \subseteq U$ and $C \subseteq W$ are subspaces, and $T$ is a linear map between $U$ and $W$, we say that $T$ \emph{is a map between $A$ and $C$} if in the case of $T: U \to W$, we have $T(A) \subseteq C$ and in the case of $T: W \to U$ we have $T(C) \subseteq A$. In this language, if $\Omega$ is a quiver, $V$ is a representation of $\Omega$ and for every $x \in \Omega_0$ we have a subspace $W(x) \subseteq V(x)$, then $W$ is a subrepresentation of $V$ if and only if for every arrow $a \in \Omega_1$ with endpoints $x, y \in \Omega_0$, $V(a)$ is a map between $W(x)$ and $W(y)$.

A \emph{strand} of a representation $V$ of a quiver $\Omega$ (without multiple edges) is a sequence $(v_0 , v_1, \ldots )$ such that there exists a joruney $x_0, x_1, \ldots$ in $\Omega$ with $v_i \in V(x_i)$ and $V(a)(v_i) = v_{i + 1}$ if there is an arrow $a : x_i \to x_{i + 1}$ in the journey and $V(a)(v_{i + 1}) = v_{i}$ if there is an arrow $a : x_{i + 1} \to x_i$ in the journey.

\subsection{Poset Filtrations and Almost Gradations}
A \emph{poset filtration} of an $R$ module $M$ consists of a partially ordered set $(P , \leq )$, and a function $F: P \to \text{Sub}(M)$ which is order-preserving, meaning $p \leq q$ implies $F_p \subseteq F_q$. Here $F_p$ denotes the image of $p \in P$ under $F$. If $P$ happens to be a totally ordered set then we say that the filtration is \emph{linear}. In this paper we will assume that all filtrations $(F , P)$ of $M$ have the property that there exists some $p \in P$ with $F_p = M$.

An \emph{almost gradation} of a poset filtration $(F , P)$ of an $R$ module $M$ is a function $C: P  \to \text{Sub}(M)$ satisfying the condition that for all $p \in P$, $F_p = F_{< p} \oplus C^p$, where we define $F_{<p} = \sum_{q < p} F_q$. 

An almost gradation $C$ of a poset filtration $(F , P)$ of an $R$ module $M$ is \emph{independent} if the family of submodules $\{C_p \mid p \in P\}$ is independent in the sense that whenever we have $c_{p_1} + \ldots + c_{p_n} = 0$ for $c_{p_i} \in C_{p_i}$ and $p_1, \ldots, p_n$ distinct elements of $P$, then $c_{p_i} = 0$ for all $1 \leq i \leq n$. Furthermore we say that $C$ \emph{spans} if for all $p \in P$, $F_p = \sum_{q \leq p} C_p$. Note that since $M = F_p$ for some $p$, in particular this implies that $M = \bigoplus_{p \in P} C_p$. 

Recall that if $(P ,  \leq_P)$ and $(Q , \leq_Q)$ are two partially ordered sets, their \emph{product} $(P ,  \leq_P) \times (Q , \leq_Q)$ is the partially ordered set $(P \times Q , \leq_{P \times Q})$ where $(p , q) \leq_{P \times Q} (p' , q')$ if and only if $p \leq_P p'$ and $q \leq_Q q'$. 

We define the \emph{intersection} of two poset filtrations $(E , P)$, and $(F , Q)$ of an $R$-module $M$, to be the poset filtration $(E \cap F , P \times Q)$ (where $P \times Q$ is given the product order) defined by $[E \cap F]_{(p , q)} = E_p \cap F_q$. One can easily check that $E \cap F : P \times Q \to \text{Sub}(M)$ is an order-preserving function, so this does indeed define a poset filtration of $M$. 

In \cite{gallup2024decompositions} Propositions 6.1 and 6.2 we proved the following result. 

\begin{proposition}\label{prop: almost gradations are independent}
Let $I$ and $J$ be totally ordered sets and $(E, I)$, $(F, J)$ two linear filtrations of an $R$-module $M$. 
    \begin{enumerate}
        \item Every almost gradation of $(E, I)$ is independent. 
        \item For any $i \in I$ and $j \in J$, we have that 
        \begin{equation}
    [E \cap F]_{< (i , j)} = E_i \cap F_{< j} + E_{< i} \cap F_j. \label{eq: intersection of linear less than}
  \end{equation} 
  and every almost gradation of $(E \cap F, I \times J)$ is independent. 
    \end{enumerate}
\end{proposition}


\section{Decompositions of Locally Finite-dimensional Representations of $A_{\infty, \infty}$ Quivers}\label{sec: decomp of a infinity}

Let $V$ be any representation of an $A_{\infty, \infty}$ quiver $\Omega$, and choosing a zero vertex and positive direction arbitrarily, label the vertices and arrows of this quiver as below.

\begin{center}
\begin{tikzcd}[arrows=-]
   \Omega \quad  = \quad  \cdots \quad \arrow[r] & y_{-1} \arrow[r, "a_{-1}"]  &   y_0 \arrow[r, "a_0"] & y_1 \arrow[r, "a_1"] & y_2 \arrow[r] & \quad \cdots 
\end{tikzcd}
\end{center}

Define a totally ordered set $I$ as follows: 
    \begin{equation*}
        m_0 \leq m_1 \leq \ldots \leq n_{-\infty} \leq \ldots \leq n_{-1} \leq n_0
    \end{equation*}

For any vertex $y_\ell$ of $\Omega$ we define two functions $L^\ell, R^\ell : I \to \text{Sub}(V(y_\ell))$ as follows. For $j \in \mathbb{Z}_{\geq 0}$, we let $L^\ell_{m_j}$ (resp. $R^\ell_{m_j}$) be the transport of the zero subspace in $V(y_{\ell - j})$ (resp. the zero subspace in $V(y_{\ell + j})$) to $V(y_\ell)$ and for $j \in \mathbb{Z}_{\leq 0}$ we define $L^\ell_{n_j}$ (resp. $R^\ell_{n_j}$) be the transport of $V(y_{\ell + j})$ (resp. $V(y_{\ell - j})$) to $V(y_\ell)$. Furthermore, we define $L^\ell_{n_{-\infty}}$ (resp. $R^\ell_{n_{-\infty}}$) to be the set of all $v \in V(y_\ell)$ such that there exists a strand $(v = v_0 , v_1, \ldots )$ of the journey $y_\ell, y_{\ell - 1}, y_{\ell - 2}, \ldots$ (respectively of the journey $y_\ell, y_{\ell + 1}, y_{\ell + 2} \ldots$).

\begin{proposition}
    $L^\ell$ and $R^\ell$ are linear filtrations of $V(y_\ell)$. 
\end{proposition}

\begin{proof}
We only show that $L$ is (the case of $R$ being similar). First of all, $L^\ell_{n_0} = V(y_\ell)$ by definition. Then for any $\ell \in \zz$, we always have $0 = L^\ell_{m_0} \subseteq L^\ell_{m_1}$ and $L^\ell_{n_{-1}} \subseteq L^\ell_{n_0} = V(y_\ell)$. Then because taking the image and inverse image of subspaces preserves inclusions, the transport also preserves inclusions. Therefore because $L^\ell_{m_i}$ and $L^\ell_{m_{i + 1}}$ are the transports of $L^{\ell - i}_{m_0} \subseteq L^{\ell - i}_{m_1}$ respectively along the journey $y_{\ell - i}, \ldots, y_\ell$ it follows that $L^\ell_{m_i} \subseteq L^\ell_{m_{i + 1}}$. A similar result shows that $L^\ell_{n_{j - 1}} \subseteq L^\ell_{n_{j}}$ for all $j \in \zz_{\leq 0}$. Finally if $v \in L^\ell_{n_{-\infty}}$ then by definition there exists a strand $(v = v_0 , v_1, \ldots )$ of the journey $y_\ell, y_{\ell - 1}, \ldots$, so given any $j \in \zz_{\leq 0}$, $v_{\ell + j} \in V(y_{\ell + j})$, and hence $v$ is in the transport of $V(y_{\ell + j})$ to $V(y_\ell)$ which is $L^\ell_{n_j}$, i.e. $L^\ell_{n_{-\infty}} \subseteq L^\ell_{n_j}$. On the other hand, given $v \in L^\ell_{m_i}$ for $i \in \zz_{\geq 0}$, by definition $v$ is in the transport of $0 \in V(y_{\ell - i})$ to $V(y_\ell)$ along the path $y_{\ell - i}, \ldots, y_\ell$, which implies that there exist $v_1, \ldots, v_{\ell - i + 1}$ such that $(v = v_0, v_1, \ldots, v_{\ell - i + 1}, 0 , 0 , \ldots)$ is a strand of the journey $y_\ell, y_{\ell - 1}, \ldots$, so $v \in L^\ell_{n_{-\infty}}$. 
\end{proof}

Our goal is to show that if $C^0$ is any almost gradation of $R^0 \cap L^0$, then we can build almost gradations of the poset filtrations $L^\ell \cap R^\ell$ for each $\ell \in \mathbb{Z}$ which are compatible with $C^0$ in the sense that we make precise in Proposition \ref{prop: Ds are subreps}. To this end, the following lemma and corollary describe how certain direct sum complements behave under subspace transport. 

\begin{lemma}\label{lem: image and inverse image complements}
    Let $T: V \to W$ be a linear map of vector spaces and suppose $A \subseteq B \subseteq V$ and $C \subseteq D \subseteq W$ are subspaces. 
    \begin{enumerate}
        \item If $Y \subseteq W$ is such that $T(B) \cap D = [T(A) \cap D + T(B) \cap C] \oplus Y$ then there exists $X \subseteq V$ such that $T$ maps $X$ isomorphically onto $Y$, and $B \cap T^{-1}(D) = [A \cap T^{-1}(D) + B \cap T^{-1}(C)] \oplus X$.

        \item If $X \subseteq V$ is such that $B \cap T^{-1}(D) = [A \cap T^{-1}(D) + B \cap T^{-1}(C)] \oplus X$ then there exists $Y \subseteq W$ such that $T$ maps $X$ isomorphically onto $Y$, and $T(B) \cap D = [T(A) \cap D + T(B) \cap C] \oplus Y$.
    \end{enumerate}
\end{lemma}

\begin{proof}

\begin{enumerate}
    \item Let $\{ y_i \mid i \in I \}$ be a basis for $Y$ and since $Y \subseteq T(B)$, for each $i \in I$ choose $x_i \in B$ with $T(x_i) = y_i$. Note that because $Y \subseteq D$, $x_i \in T^{-1}(D)$ as well. Let $X = \text{Span}(x_i \mid i \in I)$.

    First we show that $X$ and $A \cap T^{-1}(D) + B \cap T^{-1}(C)$ are independent. Suppose that we have $\sum_{i \in I} \alpha_i x_i + (a + b) = 0$ where $\alpha_i \in \mathbb{F}$, $a \in A \cap T^{-1}(D)$, and $b \in T^{-1}(C) \cap B$. Applying $T$ we obtain $\sum_{i \in I} \alpha_i y_i + T(a) + T(b) = 0$. Then $T(a) \in T(A) \cap D$ and $T(b) \in C \cap T(B)$, so by hypothesis we must have $\sum_{i \in I} \alpha_i y_i = T(a) + T(b) = 0$, hence $\alpha_i = 0$ for all $i$ by independence of $\{ y_i \mid i \in I \}$. Thus $\sum_{i \in I} \alpha_i x_i = 0$ which implies $a + b = 0$ as desired. Also note that incidentally we obtain that $\{ x_i \mid i \in I \}$ is independent, hence this set is a basis for $X$ which means $T$ maps $X$ isomorphically onto $Y$. 
     
    Now we show that $B \cap T^{-1}(D) = [A \cap T^{-1}(D) + T^{-1}(C) \cap B] + X$. Given any $z \in B \cap T^{-1}(D)$, $T(b) \in T(B) \cap D$, hence we can write $T(z) = \sum_{i \in I} \alpha_i y_i + T(a) + T(b)$ where $a \in A$, $b \in B$, and $T(a) \in T(A) \cap D$, $T(b) \in T(B) \cap C$. Then $a \in A \cap T^{-1}(D)$ and $b \in B \cap T^{-1}(C)$, hence we have $k := z - ( \sum_{i \in I} \alpha_i x_i + a + b) \in B \cap \ker T \subseteq B \cap T^{-1}(C)$. So we can write $z = \sum_{i \in I} \alpha_i x_i + a + (b + k) \in X + [A \cap T^{-1}(D) + T^{-1}(C) \cap B]$. 

    \item Let $Y = T(X)$. Then because $X \subseteq B$, we have $Y \subseteq T(B)$ and because $X \subseteq T^{-1}(D)$, we have $Y \subseteq D$. 

    First we show that $Y$ and $T(A) \cap D + T(B) \cap C$ are independent. Suppose that we have $y + (d + c) = 0$ where $y \in Y$, $d \in T(A) \cap D$, and $c \in T(B) \cap C$. Then $y = T(x)$ for some $x \in X$. Furthermore $d = T(a)$ for some $a \in A$ which then must be in $T^{-1}(D)$ as well, implying that actually $a \in A \cap T^{-1}(D)$, and $c = T(b)$ for some $b \in B$ which then must be in $T^{-1}(C)$ implying that $b \in B \cap T^{-1}(C)$. We also have $x + a + b \in \ker T \subseteq T^{-1}(C)$, so there exists $w \in T^{-1}(C)$ such that $x + a + b = w$. Note that $x + a + b \in B$, hence $w \in B \cap T^{-1}(C)$, therefore rearranging yields $x + a + (b - w) = 0$ where $x \in X$, $a \in A \cap T^{-1}(D)$, and $b - w \in B \cap T^{-1}(C)$. By hypothesis $X$ is independent from $A \cap T^{-1}(D) + B \cap T^{-1}(C)$, so $x = 0$ which implies $y = 0$ and therefore $d + c = 0$ as well, as desired. 

    Now we show that $T(B) \cap D = [T(A) \cap D + T(B) \cap C] + Y$. Given any $z \in T(B) \cap D$, there exists $v \in B$ with $T(v) = z \in D$ and hence $v \in B \cap T^{-1}(D)$ automatically. Since $B \cap T^{-1}(D) = [A \cap T^{-1}(D) + B \cap T^{-1}(C)] + X$ by hypothesis, we can write $v = x + a + b$ where $x \in X$, $a \in A \cap T^{-1}(D)$, and $b \in B \cap T^{-1}(C)$. Then $T(a) \in T(A) \cap D$, $T(b) \in T(B) \cap C$ and $T(x) \in Y$, and we have $z = T(v) = T(x) + T(a) + T(b)$, as desired. 

    Finally we must show that $T$ maps $X$ isomorphically onto $Y$. By definition $T(X) = Y$, so we need only show that $T|_X$ is injective. Suppose $x \in X$ is such that $x \in \ker T$. Then since $\ker T \subseteq T^{-1}(C)$ and since $X \subseteq B$, we have $x \in B \cap \ker T \subseteq B \cap T^{-1}(C)$. However $X$ is independent from $B \cap T^{-1}(C)$ by hypothesis, so $x = 0$. 
    
\end{enumerate}

\end{proof}

The following corollary follows immediately from Lemma \ref{lem: image and inverse image complements}.  

\begin{corollary}\label{cor: transport complements}
    Suppose that $V$ and $W$ are vector spaces and $A \subseteq B \subseteq V$ and $C \subseteq D \subseteq W$ are subspaces. If $T$ is a linear map $V \to W$ or $W \to V$ and $X \subseteq V$ is such that $B \cap D_T = [A \cap D_T + B \cap C_T] \oplus X$, then there exists $Y \subseteq W$ such that $T$ is an isomorphism between $X$ and $Y$ and $B_T \cap D = [A_T \cap D + B_T \cap C] \oplus Y$. 
\end{corollary}

We now show that the filtrations $L^\ell$ and $R^\ell$ behave well under transport. We write $n_i - 1$ (resp. $m_i - 1$) to mean the immediate predecessor of $n_i$ (resp. $m_i$) in $I$ when such an element exists and we define $n_{-\infty} - 1 = n_{-\infty}$. 

\begin{proposition}
\
\begin{enumerate}
    \item For $i > 0$, $L^{\ell + 1}_{m_i}$ (resp. $R^{\ell}_{m_i}$) is the transport of $L^{\ell}_{m_{i - 1}}$ (resp. $R^{\ell + 1}_{m_{i - 1}}$) along $V(a_\ell)$. 
    \item For $j < 0$, $L^{\ell + 1}_{n_j}$ (resp. $R^{\ell}_{n_j}$) is the transport of $L^{\ell + 1}_{n_{j + 1}}$ (resp. $R^{\ell}_{n_{j + 1}}$) along $V(a_\ell)$.
    \item $L^{\ell + 1}_{n_{-\infty}}$ (resp. $R^{\ell}_{n_{-\infty}}$) is the transport of $L^{\ell}_{n_{-\infty}}$ (resp. $R^{\ell + 1}_{n_{-\infty}}$) along $V(a_\ell)$. 
\end{enumerate}
\end{proposition}

\begin{proof}
    For (1) and (2) the claim follows from the trivial fact that the transport along a path of the transport along another path is the transport along the concatenation of the paths. We now show (3). In this case, if $w$ is in the transport of $L^\ell_{n_{-\infty}}$ along $V(a_\ell)$ then if $a_\ell$ points from $y_\ell$ to $y_{\ell + 1}$ then there exists $v_0 \in L^\ell_{n_{-\infty}}$ such that $V(a_\ell)(v_0) = w$. On the other hand if $a_\ell$ points from $y_{\ell + 1}$ to $y_\ell$ then $V(a_\ell)(w) = v_0' \in V(a_\ell)$. By definition of $L^{\ell}_{n_{-\infty}}$, in either case there exists strands $(v_0, v_1, \ldots)$ and $(v_0', v_1', \ldots)$ of the journey $y_\ell, y_{\ell - 1}, \ldots$, and then either $(w , v_0, v_1, \ldots)$ or $(w , v_0', v_1', \ldots)$ are strands of the journey $y_{\ell + 1}, y_\ell, y_{\ell - 1}, \ldots$, so $w \in L^{\ell + 1}_{n_{-\infty}}$. Conversely if $w \in L^{\ell + 1}_{n_{-\infty}}$ then there exists a strand $(w = v_0 , v_1, \ldots)$ of the journey $y_{\ell + 1}, y_\ell, y_{\ell - 1}, \ldots$ which implies that $(v_1, v_2, \ldots)$ is a strand of the journey $y_\ell, y_{\ell - 1}, \ldots$, so $v_1 \in L^{\ell}_{n_{-\infty}}$ and $w$ is in the transport of $L^{\ell}_{n_{-\infty}}$ along $V(a_\ell)$. 
\end{proof}

Now suppose that $C^0$ is any almost gradation of $L^0 \cap R^0$. Precisely this means that for each $p , q \in I$, $C^0_{p , q}$ is a direct sum complement of $[L^0 \cap R^0]_{< (p , q)} = [L^0_p \cap R^0_{< q} + L^0_{<p} \cap R^0_q]$ (see Proposition \ref{prop: almost gradations are independent}) inside of $L^0_p \cap R^0_q$. Suppose for $\ell \geq 0$ that we have an almost gradation $C^\ell$ of $L^\ell \cap R^\ell$. There are six cases to consider. We use the convention that $(-\infty) \pm 1 = -\infty$.

\begin{enumerate}
    \item If $p = m_i$ and $q = m_j$ for some $2 \leq i < \infty$ and $0 \leq j < \infty$, then $L^{\ell + 1}_{m_{i - 1}} \subseteq L^{\ell + 1}_{m_i}$ are the transports of $L^{\ell}_{m_{i - 2}}  \subseteq L^{\ell}_{m_{i - 1}}$ from $V(y_{\ell})$ to $V(y_{\ell + 1})$ and $R^{\ell}_{m_{j}} \subseteq R^{\ell}_{m_{j + 1}}$ are the transports of $R^{\ell + 1}_{m_{j - 1}} \subseteq R^{\ell + 1}_{m_{j}}$ from $V(y_{\ell + 1})$ to $V(y_{\ell})$. Since we have that 
    \begin{equation*}
        L^{\ell}_{m_{i - 1}} \cap R^{\ell}_{m_{j + 1}} = [L^{\ell}_{m_{i - 2}} \cap R^{\ell}_{m_{j + 1}} + L^{\ell}_{m_{i - 1}} \cap R^{\ell}_{m_{j}}] \oplus C^{\ell}_{m_{i - 1} , m_{j + 1}}
    \end{equation*}
    by hypothesis, it follows from Corollary \ref{cor: transport complements} that there exists $C^{\ell + 1}_{m_i , m_j} \subseteq V(y_{\ell + 1})$ such that  
    \begin{equation*}
        L^{\ell + 1}_{m_i} \cap R^{\ell + 1}_{m_j} = [L^{\ell + 1}_{m_{i - 1}} \cap R^{\ell + 1}_{m_{j}} + L^{\ell + 1}_{m_{i}} \cap R^{\ell + 1}_{m_{j - 1}}] \oplus C^{\ell + 1}_{m_i , m_j}
    \end{equation*}
    and $V(a_\ell)$ is an isomorphism between $C^\ell_{m_{i - 1} , m_{j + 1}}$ and $C^{\ell + 1}_{m_i , m_j}$. 
    
    \item If $p = n_i$ and $q = m_j$ for $-\infty \leq i \leq -1$ and $0 \leq j < \infty$, then a similar argument to that given in (1) shows there exists $C^{\ell + 1}_{n_i , m_j}\subseteq V(y_{\ell + 1})$ which is a complement of $[L^{\ell + 1} \cap R^{\ell + 1}]_{< (p , q)}$ inside of $[L^{\ell + 1} \cap R^{\ell + 1}]_{(p , q)}$ such that $V(a_\ell)$ is an isomorphism between $C^\ell_{n_{i + 1} , m_{j + 1}}$ and $C^{\ell + 1}_{n_i , m_j}$.  
    
    \item If $p = m_i$ and $q = n_j$ for $2 \leq i < \infty$ and $-\infty \leq j \leq 0$, again there exists $C^{\ell + 1}_{m_i , n_j}\subseteq V(y_{\ell + 1})$ which is a complement of $[L^{\ell + 1} \cap R^{\ell + 1}]_{< (p , q)}$ inside of $[L^{\ell + 1} \cap R^{\ell + 1}]_{(p , q)}$ such that $V(a_\ell)$ is an isomorphism between $C^\ell_{m_{i - 1} , n_{j - 1}}$ and $C^{\ell + 1}_{m_i , n_j}$.
    
    \item If $p = n_i$ and $q = n_j$ for $-\infty \leq i \leq -1$ and $-\infty \leq j \leq 0$, again there exists $C^{\ell + 1}_{n_i , n_j}\subseteq V(y_{\ell + 1})$ which is a complement of $[L^{\ell + 1} \cap R^{\ell + 1}]_{< (p , q)}$ inside of $[L^{\ell + 1} \cap R^{\ell + 1}]_{(p , q)}$ such that $V(a_\ell)$ is an isomorphism between $C^\ell_{n_{i + 1} , n_{j - 1}}$ and $C^{\ell + 1}_{n_i , n_j}$.

    \item If $p = m_1$ or $p = n_0$ we define $C^{\ell + 1}_{p , q}$ to be any direct sum complement of $[L^{\ell + 1} \cap R^{\ell + 1}]_{< (p , q)}$ inside of $[L^{\ell + 1} \cap R^{\ell + 1}]_{(p , q)}$.

    \item If $p = m_0$ then we are forced to let $C^\ell_{p , q} = 0$.

\end{enumerate}

We use a similar method to define $C^\ell_{p , q}$ for $\ell < 0$. 
Therefore we have obtained almost gradations $C^{\ell}_{p , q}$ for all $p , q \in I$ and all $\ell \geq 0$ which are compatible for varying $\ell$ in the sense that we can use them to define the following subrepresentations of $V$. For any $s , t \in \mathbb{Z}$ such that $s \leq t$, define four subrepresentations, $D^{m , m}_{s , t}$, $D^{n , m}_{s , t}$, $D^{m , n}_{s , t}$, and $D^{n , n}_{s , t}$, as follows. 

\begin{align*}
    D^{m , m}_{s , t}(y_\ell) &= \left\{ \begin{matrix} C^{\ell}_{m_{1 + \ell - s} , m_{1 + t - \ell}} & \text{ if }   & s \leq \ell \leq t \\ 0 & \text{ if } & \ell < s \text{ or } \ell > t \end{matrix} \right. 
\\ D^{m , n}_{s , t}(y_\ell) &= \left\{ \begin{matrix} C^{\ell}_{m_{1 + \ell - s} , n_{\ell - t}} & \text{ if }   & s \leq \ell \leq t \\ 0 & \text{ if } & \ell < s \text{ or } \ell > t \end{matrix} \right. 
\\ D^{n , m}_{s , t}(y_\ell) &= \left\{ \begin{matrix}  C^{\ell}_{n_{s - \ell} , m_{1 + t - \ell}} & \text{ if }   & s \leq \ell \leq t \\ 0 & \text{ if } & \ell < s \text{ or } \ell > t \end{matrix} \right. 
\\ D^{n , n}_{s , t}(y_\ell) &= \left\{ \begin{matrix} C^{\ell}_{n_{s - \ell} , n_{\ell - t}} & \text{ if }   & s \leq \ell \leq t \\ 0 & \text{ if } & \ell < s \text{ or } \ell > t \end{matrix} \right. 
\end{align*}

Furthermore define five more subrepresentations $D^{m , n}_{s , \infty}$, $D^{n , m}_{-\infty , t}$, and $D^{n , n}_{-\infty , \infty}$ as follows. 
\begin{align*}
D^{m , n}_{s , \infty}(y_\ell) &= \left\{ \begin{matrix} C^{\ell}_{m_{1 + \ell - s} , n_{-\infty}} & \text{ if }   & \ell \geq s \\ 0 & \text{ if } & \ell < s \end{matrix} \right.
\\ D^{n , n}_{s , \infty}(y_\ell) &= \left\{ \begin{matrix} C^{\ell}_{n_{s - \ell} , n_{-\infty}} & \text{ if }   & \ell \geq s \\ 0 & \text{ if } & \ell < s \end{matrix} \right.
    \\ D^{n , m}_{-\infty , t}(y_\ell) &= \left\{ \begin{matrix}  C^{\ell}_{n_{-\infty} , m_{1 + t - \ell}} & \text{ if }   & \ell \leq t \\ 0 & \text{ if } & \ell > t \end{matrix} \right. 
\\ D^{n , n}_{-\infty , t}(y_\ell) &= \left\{ \begin{matrix}  C^{\ell}_{n_{-\infty} , n_{\ell - t}} & \text{ if }   & \ell \leq t \\ 0 & \text{ if } & \ell > t \end{matrix} \right. 
\\ D^{n , n}_{-\infty , \infty}(y_\ell) &= C^{\ell}_{n_{-\infty} , n_{-\infty}}
\end{align*}

\begin{proposition}\label{prop: Ds are subreps}
For $b , c \in \{ m , n \}$ and $s , t \in \zz$, $W$ is a subrepresentation of $V$ for $W$ equal to one of $D_{s , t}^{b , c}$, $D^{m , n}_{s , \infty}$, $D^{n , m}_{-\infty , t}$, or $D^{n , n}_{-\infty , \infty}$. Furthermore if $s \leq \ell < t$ then $W(a_\ell)$ is an isomorphism and otherwise $W(a_\ell) = 0$.
\end{proposition}

\begin{proof}
    Let $\ell \in \zz$. We shall show that $V(a_\ell)$ is a map between $W(y_\ell)$ and $W(y_{\ell + 1})$. There are several cases to consider. 
    
    (\emph{Case 1}: $\ell + 1 < s$ or $\ell > t$ and any $W$). In this case $W(y_\ell) = W(y_{\ell + 1}) = 0$, hence $V(a_\ell)$ is a map between $V(a_\ell)$ is a map between $W(y_\ell)$ and $W(y_{\ell + 1})$ and $W(a_\ell) = 0$.

    (\emph{Case 2}: $\ell + 1 = s$ and $W = D_{s , t}^{b , c}$ or $W = D^{b , n}_{s , \infty}$). First of all, $W(y_\ell) = 0$. If $a_\ell : y_\ell \to y_{\ell + 1}$ then automatically $V(a_\ell)$ is a map between $W(y_\ell)$ and $W(y_{\ell + 1})$ and $W(a_\ell) = 0$. Now suppose $a_\ell : y_{\ell + 1} \to y_{\ell}$. We consider two subcases. If $W = D_{s , t}^{m , c}$ or $W = D_{s , \infty}^{m , n}$, then we have $W(y_{\ell + 1}) = C^s_{m_1, q} \subseteq L^s_{m_1} \cap R^s_{q}$ for some $q \in I$. Recall that $L^s_{m_1}$ is the transport of $0 \subseteq V(y_\ell)$ to $V(y_s)$, so $V(a_\ell)(L^s_{m_1}) = 0$, hence $V(a_\ell)(W(y_{\ell + 1})) \subseteq W(y_\ell)$ and $W(a_\ell) = 0$. On the other hand, if $W = D_{s , t}^{n , c}$ or $W = D_{s , \infty}^{n , n}$  we have $W(y_{\ell + 1}) = C^s_{n_0, q}$  for some $q \in I$. But by definition, $L^s_{n_1} = V(a_\ell)^{-1}(V(y_{\ell})) = V(y_{\ell + 1}) = L^s_{n_0}$. Thus $L^s_{n_1} \cap R^s_{q} = L^s_{n_0} \cap R^s_{q}$, and so since $L^s_{n_0} \cap R^s_{q} = [L^s_{n_1} \cap R^s_{q} + L^s_{n_0} \cap R^s_{<q}] \oplus C^s_{n_0, q}$, we have that $C^s_{n_0, q} = 0$, so indeed $V(a_\ell)(W(y_{\ell + 1})) = V(a_\ell)(0) = 0 = W(y_\ell)$, and $W(a_\ell) = 0$. 
    
    (\emph{Case 3}: $\ell = t$ and $W = D_{s , t}^{b , c}$ or $W = D^{n , c}_{-\infty , t}$). After replacing $L$ by $R$, a similar argument to that in Case 3 gives the desired results.
    
    (\emph{Case 4}: $s \leq \ell < t$ and $W = D_{s , t}^{b , c}$, $W = D_{s , \infty}^{m , n}$, $W = D_{-\infty , t}^{m , n}$, or $W = D_{-\infty , \infty}^{n , n}$). By construction $V(a_\ell)$ is an isomorphism between the following pairs of subspaces, which simultaneously shows that $V(a_\ell)$ is a map between $W(y_\ell)$ and $W(y_{\ell + 1})$ and that $W(a_\ell)$ is an isomorphism.  
    \begin{align*}
        C^\ell_{m_{1 + \ell - s}, m_{1 + t - \ell}} &\cong C^{\ell + 1}_{m_{1 + (\ell + 1) - s}, m_{1 + t - (\ell + 1)}}
        \\ C^\ell_{m_{1 + \ell - s}, n_{\ell - t}} &\cong C^{\ell + 1}_{m_{1 + (\ell + 1) - s}, n_{(\ell + 1) - t}}
        \\ C^\ell_{n_{s - \ell}, m_{1 + t - \ell}} &\cong C^{\ell + 1}_{n_{s - (\ell + 1)}, m_{1 + t - (\ell + 1)}}
        \\ C^\ell_{n_{s - \ell}, n_{\ell - t}} &\cong C^{\ell + 1}_{n_{s - (\ell + 1)}, n_{(\ell + 1) - t}}
        \\ C^\ell_{m_{1 + \ell - s}, n_{-\infty}} &\cong C^{\ell + 1}_{m_{1 + (\ell + 1) - s}, n_{-\infty}}
        \\ C^\ell_{n_{s - \ell}, n_{-\infty}} &\cong C^{\ell + 1}_{n_{s - (\ell + 1)}, n_{-\infty}}
        \\ C^\ell_{n_{-\infty}, m_{1 + t - \ell}} &\cong C^{\ell + 1}_{n_{-\infty}, m_{1 + t - (\ell + 1)}}
        \\ C^\ell_{n_{-\infty}, n_{\ell - t}} &\cong C^{\ell + 1}_{n_{-\infty}, n_{(\ell + 1) - t}}
        \\ C^\ell_{n_{-\infty}, n_{-\infty}} &\cong C^{\ell + 1}_{n_{-\infty}, n_{-\infty}}
    \end{align*}
\end{proof}

We now prove a few properties of the subrepresentations $D^{b , c}_{s , t}$ that will be needed shortly. 

\begin{lemma}\label{lem: C ell is one of the Ds}
For any $p , q \in I$, $C^\ell_{p , q}$ is equal to $W(y_\ell)$ for $W$ one of the subrepresentations $D^{b , c}_{s , t}$, $D^{b , n}_{s , \infty}$, $D^{n , c}_{-\infty , t}$, or $D^{n , n}_{-\infty , \infty}$ for $b , c \in \{ m, n \}$. 
\end{lemma}

\begin{proof}
Given $\ell \in \zz$ and $i , j \geq 0$, if we let $s = 1 + \ell - i$ and $t = \ell - 1 + j$ we obtain that $C^\ell_{m_i , m_j} = D^{m , m}_{s , t}(y_\ell)$. A similar method shows the desired result for all $p , q \in I$ and $\ell \in \zz$. 
\end{proof}

\begin{lemma}\label{lem: at most one nonzero D}
    Fix $s , t \in \zz$ with $s \leq t$. 
    
    \begin{enumerate}
        \item At most one of $D_{s , t}^{m , m}$, $D_{s , t}^{m , n}$, $D_{s , t}^{n , m}$, and $D_{s , t}^{n , n}$ is nonzero.
        \item At most one of $D^{m , n}_{s , \infty}$, $D^{n , n}_{s , \infty}$ is nonzero. 
        \item At most one of $D^{n , m}_{-\infty , t}$ and $D^{n , n}_{-\infty , t}$ is nonzero. 
    \end{enumerate} 
\end{lemma}

\begin{proof}
Suppose $a_{s - 1}$ points from $y_s$ to $y_{s - 1}$. Then $L^s_{n_1} = V(a_{s - 1})^{-1}(V(y_{s - 1})) = L^s_{n_0}$. Hence if $W$ equals $D_{s , t}^{n , c}, D_{s , \infty}^{n , n}$, we have $W(y_s) = C^s_{n_0 , q} = 0$ for all $q \in I$. Since $V(a_\ell)$ is an isomorphism between $W(y_\ell)$ and $W(y_{\ell + 1})$ for $s \leq \ell < t$, it follows that $W = 0$. On the other hand if $a_{s - 1}$ points from $y_{s - 1}$ to $y_{s}$, then $L^s_{m_1} = V(a_{s - 1})(0) = L^s_{m_0}$. Hence if $W$ equals $D_{s , t}^{m , c}, D_{s , \infty}^{m , n}$, we have $W(y_s) = C^s_{n_0 , q} = 0$ for all $q \in I$ and as above we obtain $W = 0$.

Similar arguments show that if $a_t$ points from $y_t$ to $y_{t + 1}$, then $D_{s , t}^{b , n}, D_{-\infty, t}^{n , n} = 0$ and if $a_t$ points from $y_{t + 1}$ to $y_t$ then $D_{s , t}^{b , m}, D_{-\infty, t}^{n , m} = 0$. The desired result follows. 
\end{proof}

For any $s , t \in \zz$, we define 
\begin{align*}
    E_{s , t} &= D_{s , t}^{m , m} + D_{s , t}^{m , n} + D_{s , t}^{n , m} + D_{s , t}^{n , n}, 
    \\ E_{s , \infty} &= D_{s , \infty}^{m , n} + D_{s , \infty}^{n , n},
    \\ E_{-\infty , t} &= D_{-\infty , t}^{n , m} + D_{-\infty , t}^{n , n},
    \\ E_{-\infty , \infty} &= D^{n , n}_{-\infty , \infty}
\end{align*}

\begin{proposition}
    For any $s , t \in \zz \cup \{ \pm \infty \}$, $E_{s , t}$ is an isotypic subrepresentation of $V$. It is a direct sum of a thin indecomposable representation of $\Omega$ supported on the full subquiver with vertices $\{ y_\ell \mid \ell \in [s , t] \cap \zz \}$. 
\end{proposition}

\begin{proof}
For $s , t \in \zz$, by Lemma \ref{lem: at most one nonzero D}, $E_{s , t}$ is equal to either $D_{s , t}^{m , m}$, $D_{s , t}^{m , n}$, $D_{s , t}^{n , m}$, or $D_{s , t}^{n , n}$. In any of these cases, by Proposition \ref{prop: Ds are subreps} $W(a_\ell)$ is an isomorphism if $s \leq \ell < t$ and is zero otherwise. Taking compatible bases for $E_{s , t}(y_\ell)$ for $s \leq \ell \leq t$, $E_{s , t}$ is then a direct sum of the desired thin subrepresentations, which are all isomorphic to each other. The cases with $s = -\infty$ or $t = \infty$ or both are similar.  
\end{proof}

For any $\ell \in \zz$, $C^\ell$ is an almost gradation of the intersection of two linear filtrations, so by Proposition \ref{prop: almost gradations are independent} it is always an independent almost gradation. Hence the set of subrepresentations $\{ E_{s , t} \mid s , t \in \zz \cup \{\pm \infty \} \}$ are independent. We now show that this set of subrepresentations actually spans $V$ too. Since it is not true that all (even independent) almost gradations span (see Examples \ref{ex: closed under intersection nonspanning almost gradation} and \ref{ex: finite dim nonspanning almost gradation} below), we introduce a new condition that guarantees spanning and then show that the almost gradations $C^\ell$ satisfy this condition. 

\begin{definition}
We say that a linear filtration $F$ of $V$ is \emph{closed under intersection} if for all subsets $A \subseteq I$ there exists $i \in I$ such that $i \leq a$ for all $a \in A$ and $F_i = \bigcap_{a \in A} F_a$. 
\end{definition}

\begin{remark}\label{rem: conditions to check for closed under intersection}
\
\begin{enumerate}
    \item It suffices to check the condition for subsets $A \subseteq I$ which are \emph{saturated}, meaning that if $i \in I$ has the property that $a \leq i \leq b$ for $a , b \in A$ then $i \in A$.
    \item If $A$ contains its infimum then the condition is automatically satisfied by taking $i$ to be this infimum.
\end{enumerate} 
\end{remark}

\begin{lemma}\label{lem: closed under intersection in the fd case}
    If $V$ is a finite-dimensional vector space, $F: I \to \operatorname{Sub}(V)$ is a linear filtration, $A \subseteq I$ is any subset, then there exists $a_0 \in A$ such that $F_{a_0} = F_b$ for all $b \in A$ with $b \leq a_0$, and hence $\bigcap_{a \in A} F_a = F_{a_0}$.
\end{lemma}

\begin{proof}

Since $V$ is finite-dimensional, the set of subspaces $\{ F_a \mid a \in A \}$ has a minimal element, say $F_{a_0}$. Then given any $b \in A$ with $b \leq a_0$ we have $F_b \subseteq F_{a_0}$, so by minimality it must be that $F_b = F_{a_0}$ as desired.
\end{proof}

\begin{proposition}\label{prop: fd intersection closed spans linear}
If $I$ is a totally ordered set and $F : I \to \text{Sub}(V)$ is a linear filtration of a finite-dimensional vector space $V$ which is closed under intersection, then every almost gradation of $F$ spans $F$. 
\end{proposition}

\begin{proof}
Let $C$ be an almost gradation of $F$, let $p \in I$ and let $v \in F_p$. Consider the set $A = \{ a \leq p \mid v \in F_a \}$. By hypothesis there exists $i_1 \in I$ such that $i_1 \leq a$ for all $a \in A$ and $F_{i_1} = \bigcap_{a \in A} F_a$. Then $v \in F_{i_1}$ and by definition of an almost gradation we have $F_{i_1} = C_{i_1} \oplus F_{< i_1}$ so we can write $v = c_{i_1} + v_{i_2}$ where $c_{i_1} \in C_i$ and $v_{i_2} \in F_{<i}$. We claim that $c_{i_1} \neq 0$. Indeed, if this is the case then $v \in F_{< i_1} = \bigcup_{j < i_1} F_j$ so in fact there exists some $j < i_1$ such that $v \in F_j$. But this means that $j \in A$ contradicting that $i_1 \leq a$ for all $a \in A$, so it follows that $c_{i_1} \neq 0$. If $v_{i_2} = 0$, we're done. If not, replace $v$ by $v_{i_2}$ and repeat this process. If at some point $v_{i_n} = 0$, again we're done. If this never happens, we obtain a sequence $i_1 > i_2 > \ldots$ with $c_{i_n} \in C_{i_n} \smallsetminus \{ 0 \}$. Since $C$ is an almost gradation of a linear filtration, by Proposition \ref{prop: almost gradations are independent} it is independent, i.e. the family of subspaces $C_{i_1}, C_{i_2} \ldots$ is independent, hence $c_{i_1}, c_{i_2}, \ldots$ is an independent set, contradicting that $V$ is a finite-dimensional vector space. 
\end{proof}

\begin{corollary}\label{cor: closed under cap implies spanning ag}
If $I$ and $J$ are totally ordered sets and $E: I \to \text{Sub}(V)$, $F: J \to \text{Sub}(V)$ are filtrations of a finite-dimensional vector space $V$ which are both closed under intersection, then every almost gradation of $E \cap F: I \times J \to \text{Sub}(V)$ spans $E \cap F$. 
\end{corollary}


\begin{proof}
Let $C : I \times J \to \text{Sub}(V)$ be an almost gradation of $E \cap F$ and let $v \in E_a \cap F_b$. Since $E$ is closed under intersection, there exists $i \in I$ such that $v \in E_i$ but $v \notin E_{< i}$. Then the residue $\overline{v} \in E_i / E_{<i}$ is nonzero. 

We claim that the function $\overline{F}^i : J \to \text{Sub}(E_i / E_{<i})$ defined by $j \mapsto \overline{E_i \cap F_j}$ is a filtration which is closed under intersection (here $\overline{E_i \cap F_j}$ denotes the image of $E_i \cap F_j$ in the quotient $E_i / E_{<i}$). First of all, this function is a filtration because on the one hand it is order-preserving since both intersecting with $E_i$ and projecting to the quotient preserve containment, and on the other hand since $(F , J)$ is a filtration of $V$ and hence there exists $j_1 \in J$ such that $F_{j_1} = V$, and hence $\overline{F}^i_{j_1} = \overline{E_i \cap F_{j_1}} = \overline{E_i} = E_i / E_{<i}$. Second of all, this filtration is closed under intersection because given any subset $A \subseteq J$, since $F$ is closed under intersection there exists $j \in J$ with $j \leq a$ for all $a \in A$ with the property that $F_j = \bigcap_{a \in A} F_a$, which implies that $E_i \cap F_j = \bigcap_{a \in A} E_i \cap F_a$. Note that we automatically have $\overline{E_i \cap F_j} \subseteq \bigcap_{a \in A} \overline{E_i \cap F_a}$. However because $V$ is finite-dimensional, by Lemma \ref{lem: closed under intersection in the fd case}, there exists $a_0 \in A$ such that $F_j = F_b$ for all $b \leq a_0$. Hence in fact we have $F_j = \bigcap_{a \in A} F_a = F_{a_0}$, which implies $E_i \cap F_j = E_i \cap F_{a_0}$, so $\overline{E_i \cap F_j} = \overline{E_i \cap F_{a_0}}$, and thus $\overline{E_i \cap F_j} = \bigcap_{a \in A} \overline{E_i \cap F_a}$ as desired. 

We now show that $\overline{E_i \cap F_{<j}} = \overline{F}^i_{<j}$. Note that if $\{ j' \in J \mid j' < j \} = \emptyset$, then $F_{<j} = 0$ and $\overline{F}^i_{<j} = 0$, so this equality is trivial. On the other hand, if $\{ j' \in J \mid j' < j \} \neq \emptyset$, then because $J$ is totally ordered, we have that $F_{<j} = \bigcup_{j' < j} F_{j'}$ and similarly $\overline{F}^i_{<j} = \bigcup_{j' < j} \overline{E_i \cap F_{j'}}$, therefore $E_i \cap F_{<j} = E_i \cap \bigcup_{j' < j} F_{j'} = \bigcup_{j' < j} E_i \cap F_{j'}$, and since the projection map $E_i \to E_i / E_{<i}$ preserves unions, the desired equality follows. 

Now we claim that the function $\overline{C}^i : J \to \text{Sub}(E_i / E_{<i})$ which sends $j \mapsto \overline{C_{i , j}}$ is an almost gradation of $\overline{F}^i$. One the one hand, given any $j \in J$ and $\overline{v} \in \overline{F}^i_j = \overline{E_i \cap F_j}$, because $C$ is an almost gradation of $E \cap F$, by Equation \ref{eq: intersection of linear less than} we have $E_i \cap F_j = [E_{<i} \cap F_j + E_i \cap F_{<j}] \oplus C_{i , j}$, hence we can write $v = c + u + w$ where $c \in C_{i , j}$, $u \in E_{<i} \cap F_j$, and $w \in E_i \cap F_{<j}$. Taking residues mod $E_{<i}$ yields $\overline{v} = \overline{c} + \overline{w}$, where $\overline{c} \in \overline{C}^i$ and $\overline{w} \in \overline{E_i \cap F_{<j}}$. Recall that we just showed this latter space is equal to $\overline{F}^i_{<j}$, hence we have $\overline{F}^i_j = \overline{F}^i_{<j} + \overline{C}^i$. On the other hand, we must show that $\overline{F}^i_{<j}$ and $\overline{C}^i$ are independent. Suppose that $\overline{w} \in \overline{F}^i_{<j}$ and $\overline{c} \in \overline{C}^i$ satisfy $\overline{w} + \overline{c} = 0$. Then $w \in E_i \cap F_{<j}$ and $c \in C_{i , j}$ and $w + c \in E_{<i} \cap F_j$. By independence of $C_{i , j}$ and $E_{<i} \cap F_j + E_i \cap F_{<j}$ we have $c = 0$, hence $\overline{c} = 0$ and $\overline{w} = 0$ as desired.

Since $\overline{C}^i$ is an almost gradation of $\overline{F}^i$, by Proposition \ref{prop: fd intersection closed spans linear}, $\overline{C}^i$ spans, so since $\overline{v} \in \overline{E_i \cap F_b} = \overline{F}^i_b$, we can write $\overline{v} = \overline{c_{i, j_1}} + \ldots + \overline{c_{i , j_{n_i}}}$ where $\overline{c_{i , j_\ell}} \in \overline{C_{i , j_\ell}}$, $j_\ell \leq b$, and we can assume that $c_{i, j_1}, \ldots, c_{i , j_{n_i}}$ are nonzero because $v \notin E_{<i}$. Therefore $v = c_{i, j_1} + \ldots + c_{i , j_{n_i}} + v'$ where $v' \in E_{< i}$. Because $v , c_{i , j_\ell} \in F_b$, we have $v' \in F_b$ as well. If $v' = 0$ we're done. If not, we may replace $v$ by $v'$ and repeat the process. As in the case of Proposition \ref{prop: fd intersection closed spans linear}, because $V$ is finite-dimensional, at some point in this process we must obtain $v' = 0$, so $v \in \sum_{(i , j) < (a, b)} C_{i , j}$ as desired.

\end{proof}

\begin{remark}
In \cite{gallup2024decompositions} we showed that if $P$ is any well-founded partially ordered set then every almost gradation of any poset filtration $F$ of $P$ spans $F$.  
\end{remark}

We now give two examples that show the necessity of the conditions of $V$ being finite-dimensional and $F$ being closed under intersection in Proposition \ref{prop: fd intersection closed spans linear}.  

\begin{example}\label{ex: closed under intersection nonspanning almost gradation}
Let $V$ be the uncountable dimensional vector space $\{ ( a_0, a_1, \ldots ) : a_i \in \mathbb{F} \}$, let $I = \mathbb{Z}_{\leq 0} \cup \{ -\infty \}$ be the totally ordered set with the usual order. Define $F: I \to \text{Sub}(V)$ by setting $F_i = \{ ( a_0, a_1, \ldots ) \in V \mid a_0 = \ldots = a_{i - 1} = 0 \}$ for $i \in \mathbb{Z}_{\leq 0}$ and $F_{-\infty} = 0$. Then $F_i / F_{< i} \cong \mathbb{F}$, hence if $C$ is any almost gradation of $F$, the sum of all $C_i$ has countable dimension, and hence cannot span $F_0$. Note that in this case the only saturated subset of $I$ which is does not contain its infimum is $\mathbb{Z}_{\leq 0}$, but the element $-\infty \in I$ satisfies the property that $F_{-\infty} = \bigcap_{ i \in \mathbb{Z}_{\leq 0} } F_i$, hence $F$ is closed under intersection. 
\end{example} 

\begin{example}\label{ex: finite dim nonspanning almost gradation}
     Let $V$ be a $1$-dimensional vector space. Define $I = \mathbb{Z}_{\leq 0} \cup \{ -\infty \}$ to be the totally ordered set with the usual order, and $F: I \to \text{Sub}(V)$ to send $i \mapsto V$ if $i \in \mathbb{Z}_{\leq 0}$ and $-\infty \mapsto 0$, then every almost gradation $C$ of $F$ has the property that $C_i = 0$ for all $i \in I$, so clearly $C$ doesn't span $F$. Note that although $V$ is finite-dimensional, this filtration is not closed under intersection since the only lower bound for the subset $\zz_{\leq 0} \subseteq I$ is $-\infty$, yet $F_{-\infty} = 0 \neq \bigcap_{i \in \zz_{\leq 0}} F_i = V$. 
\end{example}

\begin{lemma}\label{lem: intersection of transports is set of strands}
    If $V$ is locally finite-dimensional, then for all $i$ we have that $L^i_{n_{-\infty}} = \bigcap_{j \in \zz_{\leq 0}} L^i_{n_j}$, and similarly for the filtration $R_i$.
\end{lemma}

\begin{proof}
    We prove the statement for $L$ only, the case of $R$ being similar. The containment $(\subseteq)$ is clear. For ease of notation, let $T^i_{n_{-\infty}} = \bigcap_{j \in \zz_{\leq 0}} L^i_{n_j}$. Since $V(y_i)$ is finite-dimensional, the decreasing chain of subspaces $L^i_{n_0} \supseteq L^i_{n_{-1}} \supseteq L^i_{n_{-2}} \supseteq \ldots$ must stabilize, say at $J_0$, so that $L^i_{n_j} = L^i_{n_{j - 1}}$ for all $j \leq J_0$. Therefore $T^i_{n_{-\infty}} = L^i_{n_{J_0}}$. But notice that the same statement holds for $L^{i - {J_0}}$, i.e. there exists some $J_1$ such that $T^{i - J_0}_{n_{-\infty}} = L^{i - J_0}_{n_{J_1}}$. But the transport of $L^{i - J_0}_{n_{J_1}}$ to $x_i$ is equal to $L^i_{n_{J_0 - J_1}}$ which we have just shown is equal to $T^i_{n_{-\infty}}$. Continuing in this way we obtain a sequence $J_0, J_1, J_2, \ldots$ in in $\mathbb{N}$ such that $T^i_{n_{-\infty}}$ is equal to the transport of $T^{i - J_0 - J_1 - \ldots - J_\ell}_{n_{-\infty}}$ to $y_i$. So given any $v = w_0 \in T^i_{n_{-\infty}}$ there exists $w_1 \in T^{i- J_0}_{n_{-\infty}}$ whose transport to $y_i$ is $w_0$, and similarly $w_2 \in T^{i - J_0 - J_1}_{n_{-\infty}}$ whose transport to $y_{i - J_0}$ is $w_1$. So we obtain a sequence $v = w_0, w_1, w_2, \ldots$ which is a subsequence of a strand $(v_0, v_1, \ldots)$ which is contained in $L^i_{n_{-\infty}}$ by definition. 
\end{proof}

\begin{corollary}\label{cor: L and R closed under cap}
If $V$ is a locally finite-dimensional representation of an $A_{\infty, \infty}$ quiver $\Omega$ then for all $\ell \in \zz$, the filtrations $L^\ell$ and $R^\ell$ of $V(y_\ell)$ are closed under intersection. 
\end{corollary}

\begin{proof}
We prove only the case for $L^\ell$, the case of $R^\ell$ being similar. By Remark \ref{rem: conditions to check for closed under intersection} it suffices to check only saturated subsets of $I$ which do not have a lower bound. The only such subset is $\{ \ldots, n_{-2}, n_{-1}, n_0 \}$. On the one hand, $n_{-\infty} \leq n_i$ for all $i \in \zz_{\leq 0}$, and on the other hand by Lemma \ref{lem: intersection of transports is set of strands} we have that $L^\ell_{n_{-\infty}} = \bigcap_{j \in \zz_{\leq 0}} L^\ell_{n_j}$ as desired. 
\end{proof}

\begin{theorem}
    Let $V$ be a locally finite-dimensional representation of an $A_{\infty, \infty}$ quiver and let $C^0$ be any almost gradation of $L^0 \cap R^0$. Let $D^{b , c}_{s , t}$ and $E_{s, t}$ be the corresponding subrepresentations of $V$ which exist by Proposition \ref{prop: Ds are subreps}. Then $V = \bigoplus_{s \leq t} E_{s , t}$.
\end{theorem}

\begin{proof}
For every $\ell \in \zz$, by Corollary \ref{cor: L and R closed under cap}, $L^\ell$ and $R^\ell$ are closed under intersection, hence by Corollary \ref{cor: closed under cap implies spanning ag}, $C^\ell$ spans $V(y_\ell)$. By Lemma \ref{lem: C ell is one of the Ds}, for $s , t \in \zz$ and $b , c \in \{ m , n \}$, the subrepresentations $D^{b , c}_{s , t}$, $D^{b , n}_{s , \infty}$, $D^{n , c}_{-\infty , t}$, and $D^{n , n}_{-\infty , \infty}$ span $V$. Hence clearly the subrepresentations $\{ E_{s , t} \mid -\infty \leq s \leq t \leq \infty \}$ span $V$ too. 

On the other hand, by Proposition \ref{prop: almost gradations are independent} the family of subspaces $\{C^{\ell}_{p , q} \mid p , q \in I \}$ is independent, hence the family of subrepresentations $\{ D^{b , c}_{s , t}, D^{b , n}_{s , \infty}, D^{n , c}_{-\infty , t}, D^{n , n}_{-\infty , \infty} \mid s , t \in \zz, b , c \in \{ m , n \} \}$ is also independent. By Lemma \ref{lem: at most one nonzero D} for any $-\infty \leq s \leq t \leq \infty$, $E_{s , t}$ is equal to one of $D^{m ,m}_{s , t}$, $D^{m , n}_{s , t}$, $D^{n , m}_{s , t}$, $D^{n , n}_{s , t}$, $D^{m , n}_{s , \infty}$, $D^{n , n}_{s , \infty}$, $D^{n , m}_{-\infty , t}$, $D^{n , n}_{-\infty , t}$, $D^{n , n}_{-\infty , \infty}$, hence the family $\{ E_{s , t} \mid -\infty \leq s \leq t \leq \infty \}$ is independent as well.
\end{proof}

As a consequence, we recover the following analog of \cite[Theorem 1]{gallup2024decompositions} for locally finite-dimensional representations of a (not necessarily eventually outward) $A_{\infty, \infty}$ quiver which was proved in \cite{bautista-liu-paquette2011}.  

\begin{corollary}
Let $\Omega$ be a (not necessarily eventually outward) $A_{\infty, \infty}$ quiver. Then the category $\operatorname{rep}(\Omega)$ of locally finite-dimensional representations of $\Omega$ is infinite Krull-Schmidt. Furthermore, the indecomposables are all thin and taking supports gives a bijection between these indecomposables and the connected subquivers of $\Omega$. 
\end{corollary}



\section{Locally Finite-dimensional Indecomposables of Finitely Branching Tree Quivers}\label{sec: tails}

Suppose that $x_0, x_1, x_2, \ldots$ is a tail of a quiver $\Omega$ and that $V$ is a locally finite-dimensional representation of $\Omega$. We use $\Omega$ to define an $A_{\infty, \infty}$ quiver $\Omega'$ whose vertex set is $\{ y_n \mid n \in \mathbb{Z} \}$ and where there is an arrow from $y_n$ to $y_{n - 1}$ for $n \leq 0$ and if $n, m \geq 0$ then the set of arrows from $y_n$ to $y_m$ is exactly the set of arrows from $x_n$ to $x_m$.

\begin{example}\label{ex: D infinity tail}
Let $\Omega$ be the following representation of $D_\infty$ with all arrows pointing to the right. 

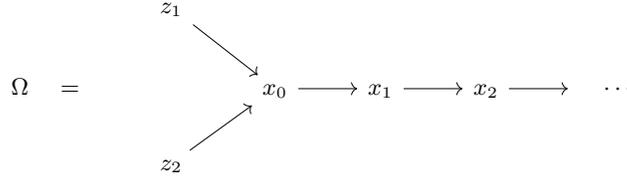
\begin{center}
 \begin{tikzcd}
 & z_1 \arrow[rd] & & & &  \\
\Omega \quad = &  & x_0  \arrow[r] & x_1 \arrow[r] &
  x_2 \arrow[r] & \quad \cdots \quad  \\
& z_2 \arrow[ru]  & & & & 
\end{tikzcd}
\end{center}

Then $x_0, x_1, x_2, \ldots$ is a tail of $\Omega$ and the corresponding $\Omega'$ is the following representation of $A_{\infty, \infty}$ with all arrows pointing away from $y_0$. 

\begin{center}
\begin{tikzcd}
\Omega' = \quad  \cdots \quad & y_{-2} \arrow[l] & y_{-1} \arrow[l]  &   y_0 \arrow[l] \arrow[r] & y_1  \arrow[r] & y_2 \arrow[r] & \quad \cdots 
\end{tikzcd}
\end{center}
\end{example}

We also use $V$ to define a representation $V'$ of $\Omega'$ by setting $V'(y_n) = V(x_n)$ for $n \geq 0$ and $V'(y_m) = V(x_0)$ for $m \leq 0$. Furthermore if $a$ is an arrow in $\Omega$ which connects $y_{n}$ and $y_m$ for $n , m \geq 0$ then by definition there is a corresponding arrow connecting $x_n$ and $x_m$ and we define $V'(a) = V(a)$. If $a$ is an arrow connecting $y_n$ and $y_m$ for $n , m \leq 0$ we define $V'(a)$ to be the identity map (since in that case $V'(y_n) = V'(y_m) = V(x_0)$. 

\begin{example}
Let $\Omega$, the tail $x_0, x_1, \ldots$, and $\Omega'$ be as in Example \ref{ex: D infinity tail} and let $V$ be the indecomposable representation of $\Omega$ pictured below which has dimension $2$ at $x_0$, dimension $1$ at $x_1$, and is zero at $x_n$ for $n \geq 2$ (note that this representation necessarily has dimension $1$ at $z_1$ and $z_2$).  

\begin{center}
 \begin{tikzcd}
 & \mathbb{F} \arrow[rd] & & & &  \\
V \quad = &  & \mathbb{F}^2  \arrow[r] & \mathbb{F} \arrow[r] &
  0 \arrow[r] & \quad \cdots \quad  \\
& \mathbb{F} \arrow[ru]  & & & & 
\end{tikzcd}
\end{center}

Then the representation $V'$ of $\Omega'$ is the same as $V$ for all vertices and arrows to the right of $y_0$, and to the left is just infinitely many copies of a two-dimensional vector space with all arrows mapping to isomorphisms as shown below.   

\begin{center}
\begin{tikzcd}
V' = \quad  \cdots \quad & \mathbb{F}^2 \arrow[l, "\sim" '] & \mathbb{F}^2 \arrow[l, "\sim" ']  &  \mathbb{F}^2 \arrow[l, "\sim" '] \arrow[r] & \mathbb{F}  \arrow[r] & 0 \arrow[r] & \quad \cdots 
\end{tikzcd}
\end{center}
\end{example}

Given any locally finite-dimensional representation $V$ of a quiver $\Omega$ with a tail $x_0, x_1, \ldots$, let $\Omega'$ and $V'$ be as described above.

\begin{corollary}\label{cor: maps on tails must be inj or surj}
    Suppose $\Omega$ is a quiver, $z_0, z_1, \ldots$ is a tail of $\Omega$, and $V$ is a locally finite-dimensional indecomposable representation of $\Omega$ with $V(z_0) \neq 0$. Let $e$ be the unique arrow between $x_0: = z_n$ and $x_1:= z_{n + 1}$ for some $n \geq 0$. If $e$ points $x_0 \to x_1$ then $V(e)$ must be surjective, and if $e$ points $x_1 \to x_0$ then $V(e)$ must be injective. 
\end{corollary}

\begin{proof}
    Without loss of generality we may assume the tail is $x_0, x_1, \ldots$ and that $e$ is an arrow between $x_0$ and $x_1$. Consider the $A_{\infty, \infty}$ quiver $\Omega'$ and its representation $V'$ associated to the quiver $\Omega$, its representation $V$, and its tail $x_0, x_1, \ldots$ defined above. Then as in Section \ref{sec: decomp of a infinity}, we have the poset filtrations $L^\ell \cap R^\ell$ of $V'(y_\ell)$ for all $\ell \in \zz$ and letting $C^0$ be any almost gradation of $L^0 \cap R^0$, the results of Section \ref{sec: decomp of a infinity} give associated subrepresentations $E_{s, t}$ for $s, t \in \zz \cup \{ \pm \infty \}$ such that $V' = \bigoplus_{s \leq t} E_{s , t}$.

    Notice that because for all $\ell < 0$, $V'(a_\ell)$ is an isomorphism, $L^0_{m_i} = 0$ for all $i \in \zz_{\geq 0}$ and $L^0_{n_j} = V'(y_0)$ for all $j \in \zz_{<0} \cup \{ -\infty \}$. Therefore we must have $C^0_{p , q} = 0$ unless $p = n_{-\infty}$, and hence if $s \leq 0$ then $E_{s , t} = 0$ unless $s = -\infty$. Therefore we have $V' = \bigoplus_{1 \leq s \leq t} E_{s , t} \oplus \bigoplus_{t} E_{-\infty, t}$. 

    If $s \geq 1$, then $E_{s , t}(y_\ell) = 0$ for $\ell \leq 0$, hence for such $s$, $E_{s , t}$ is a subrepresentation of $V$, and the family of subrepresentations $\{ E_{s , t} \mid s \geq 1 \}$ of $V$ is still independent, thus $U := \bigoplus_{1 \leq s \leq t} E_{s , t}$ is a subrepresentation of $V$. Let $W$ be the subrepresentation of $V$ defined by $W(x) = V(x)$ if $x$ is not equal to $x_0, x_1, x_2, \ldots$ and $W(x_i) = \bigoplus_{t} E_{-\infty, t}(x_i)$ for $i \in \zz_{\geq 0}$. Notice that in fact $W(x_0) = V(x_0)$. We claim that $W$ is actually a subrepresentation, i.e. that $V(a)$ is a map between $W(x)$ and $W(y)$ for all all arrows whose vertices are $x$ and $y$. If $x, y \in \{ x_1, x_2, \ldots \}$, the claim follows because $\bigoplus_{t} E_{-\infty, t}$ is a subrepresentation of $V'$ and if at most one of $x , y$ is in $\{ x_0, x_1, \ldots \}$ the claim follows trivially since in this case $W(x) = V(x)$ and $W(y) = V(y)$. Therefore we must only consider the case when $a = e$ which is an arrow between $x_0$ and $x_1$. 

    If $e$ points opposite the direction of the tail the result is trivial. since $W(x_0) = V(x_0)$. On the other hand, if $e$ points in the direction of the tail, then we have $L^1_{m_i} = 0$ for all $i \in \zz_{\geq 0}$, $L^1_{n_{-\infty}} = L^1_{n_j} = \im V(e)$ for $j \in \zz_{<0}$, and $L^1_{n_0} = V(x_1)$. Therefore in particular $\im V(e) = L^1_{n_1} = L^1_{n_{-\infty}} = \bigoplus_{t} E_{-\infty, t}(x_1)$. 

    Hence we have two subrepresentations $U, W$ of $V$ and we claim $V = U \oplus W$. Indeed for any $x \notin \{ x_1 , x_2, \ldots \}$, $U(x) = 0$ and $W(x) = V(x)$, so the result follows, and if $x \in \{ x_1, x_2, \ldots \}$ then $U(x) = \bigoplus_{1 \leq s \leq t} E_{s , t}(x)$ and $W(x) = \bigoplus_{t} E_{-\infty, t}(x)$, and we have $V(x) = V'(x) = \bigoplus_{1 \leq s \leq t} E_{s , t}(x) \oplus \bigoplus_{t} E_{-\infty, t}(x)$ as desired. 

    Now if $e$ points in the direction of the tail but is not surjective then because as above we had $\im V(e) = W(x_1)$, it must be that $U(x_1) \neq 0$. On the other hand, if $e$ points in the opposite direction of the tail and is not injective then we have $L^1_{n_{-\infty}} = L^1_{n_j} = V(x_1)$ for all $j \in \zz_{\leq 0}$ and $L^1_{m_1} = L^1_{m_i} = \ker V(e)$ for all $i \in \zz_{>0}$ while $L^1_{m_0} = 0$. Therefore $\ker V(e) = U(x_1)$ which is nonzero by hypothesis. 
    
    In either case, $U$ is not the zero representation. Since $V$ is indecomposable, it must be that $W$ is the zero representation. But since $z_0 \notin \{ x_1, x_2, \ldots \}$, this means that $V(z_0) = W(z_0) = 0$, contradiction.
\end{proof}

\begin{corollary}\label{cor:dim does not increase on tails ind}
     If $\Omega$ is a quiver, $x_0, x_1, \ldots$ is a tail of $
     \Omega$, and $V$ is a locally finite-dimensional indecomposable representation of $\Omega$ such that $V(x_0) \neq 0$ then it cannot be that $\dim V(x_n) < \dim V(x_{n + 1})$ for any $n \geq 0$.  
\end{corollary}

\begin{proof}
    Suppose $\dim V(x_n) < \dim V(x_{n + 1})$ for some $n \geq 0$ and let $a$ be the arrow connecting $x_n$ and $x_{n + 1}$. Then $V(a)$ is either not surjective, if $a : x_{n} \to x_{n + 1}$ or not injective, if $a: x_{n + 1} \to x_n$. This contradicts Corollary \ref{cor: maps on tails must be inj or surj}.  
\end{proof}

\begin{theorem}\label{thm: finitely branching tree quiver indecomposable implies flei}
    If $\Omega$ is a finitely branching tree quiver and $V$ is a locally finite-dimensional indecomposable representation of $\Omega$ then $V$ is FLEI. 
\end{theorem}

\begin{proof}
Since $\Omega$ is a finitely branching tree quiver, it is the union of a finite quiver and finitely many infinite tails. If on a tail $x_0, x_1, \ldots$ we have $V(x_n) \neq 0$ for some $n$, then by \cite[Lemma 4.2]{gallup2023gabriel} it must be that $V(x_0) \neq 0$. Since $V$ is indecomposable, by Corollary \ref{cor:dim does not increase on tails ind} the dimension can never increase along this tail, and since $V$ is locally finite-dimensional, the dimension must stabilize. However by Corollary \ref{cor: maps on tails must be inj or surj}, all maps on this tail are either injective or surjective, so after the dimensions stabilize all maps must be isomorphisms. Since there are only finitely many such tails, for all but finitely many arrows $a \in \Omega_1$ we have that $V(a)$ is an isomorphism, and hence $V$ is FLEI by definition.  
\end{proof}


\section{Infinite Gabriel's Theorem for Locally Finite-Dimensional Representations}\label{sec: proof}

\begin{theorem}[Locally Finite-Dimensional Infinite Gabriel's Thm.]
Let $\Omega$ be a connected quiver. The category $\operatorname{rep}(\Omega)$ has unique representation type if and only if $\Omega$ is a generalized ADE Dynkin quiver (see Figure \ref{fig: generalized ade quivers}) and in this case, taking dimension vectors gives a bijection between the set of isomorphism classes of indecomposable representations and the positive roots of $\Omega$ (as defined in \cite{gallup2023gabriel}). 
\end{theorem}

\begin{proof}
Suppose that $\Omega$ is a generalized ADE Dynkin diagram. If $\Omega$ has underlying graph $A_n, D_n, E_6, E_7,$ or $E_8$ then the desired conclusion is obtained by appealing to \cite{gabriel1972} and \cite{ringel2016} as in the proof of \cite[Theorem 2]{gallup2023gabriel}. If $\Omega$ has underlying graph $A_{\infty}$, $A_{\infty, \infty}$, or $D_\infty$, then $\Omega$ is in particular a finitely branching tree quiver, hence by Theorem \ref{thm: finitely branching tree quiver indecomposable implies flei}, every locally finite-dimensional indecomposable representation of $\Omega$ is FLEI. In \cite[Section 3]{gallup2023gabriel}, we classified the FLEI indecomposables of $A_{\infty, \infty}$ and $D_{\infty}$ and showed in \cite[Proposition 7.4]{gallup2023gabriel}\footnote{Note that although ``eventually outward'' is a hypothesis of Proposition 7.4, the condition is not necessary and indeed not even used in the proof.} that taking dimension vectors gives a bijection between these indecomposables and the positive roots of $\Omega$ in this case.

Conversely, if the underlying graph of $\operatorname{rep}(\Omega)$ is not a generalized ADE Dykin diagram, then as in the proof of \cite[Theorem 2]{gallup2023gabriel}, we note that $\Omega$ contains a finite subquiver $\Omega'$ which is not an ADE Dynkin diagram, and therefore has two non-isomorphic indecomposable representations with the same dimension vector, which we emphasize here can be taken to be \emph{finite-dimensional}. Extending these representations by zero outside of $\Omega'$ gives two non-isomorphic locally finite-dimensional indecomposable representations of $\Omega$ with the same dimension vector.  
\end{proof}

\section{Future Work}

In \cite{gallup2023gabriel} we showed that when $\Omega$ is an eventually outward finitely branching tree quiver, the category $\operatorname{Rep}(\Omega)$ is infinite Krull-Schmidt and all indecomposables are FLEI. In \cite{bautista-liu-paquette2011} it was proved that for any quiver $\Omega$, the category $\operatorname{rep}(\Omega)$ is infinite Krull-Schmidt, and in Theorem \ref{thm: finitely branching tree quiver indecomposable implies flei} we showed that if $\Omega$ is a finitely branching tree quiver then all indecomposables are also FLEI. 

This leaves the question: can we classify the indecomposable representations of finitely branching tree quivers which are not eventually outward? As we showed, any such indecomposable which is locally finite-dimensional is automatically FLEI, but are there not locally finite-dimensional indecomposables in this case? This question is even interesting in the case when the underlying graph of $\Omega$ is $A_{\infty , \infty}$, and is made potentially more difficult by the fact that $\operatorname{Rep}(\Omega)$ is not infinite Krull-Schmidt in this case, as we showed in \cite[Section 8]{gallup2024decompositions}.

\end{document}